\documentclass{article}
\usepackage[utf8]{inputenc}

\usepackage{graphicx,graphics}

\usepackage{rotating}
\usepackage[twoside]{geometry}
\usepackage{epsfig}
\usepackage{indentfirst}
\usepackage[usenames,dvipsnames,svgnames,table]{xcolor}
\usepackage{cmap}

\usepackage{graphicx,graphics}

\usepackage{hyperref}
\hypersetup{
pdftitle={Fractal generic},    
pdfauthor={},    
pdfnewwindow=true,      
colorlinks=true,      
linkcolor=black,         
citecolor=blue}

\usepackage{amssymb,amsfonts,amstext,amsthm}
\usepackage{mathrsfs}
\usepackage[intlimits]{amsmath}
\newcommand{\OBSI}{\begin{remark}\begin{rm}}
\newcommand{\OBSF}{\end{rm}\end{remark}}
\newcommand{\DEFI}{\begin{definition}\begin{rm}}
\newcommand{\DEFF}{\end{rm}\end{definition}}

\newcommand{\be}{\begin{eqnarray}}
\newcommand{\en}{\end{eqnarray}}
\newcommand{\bee}{\begin{eqnarray*}}
\newcommand{\ene}{\end{eqnarray*}}

\DeclareMathOperator*{\essinf}{ess.inf}
\DeclareMathOperator*{\esssup}{ess.sup}

\DeclareMathOperator*{\supp}{supp}

\newtheorem{definition}{\bf Definition}[section]
\newtheorem{proposition}{\bf Proposition}[section]

\newtheorem{lemma}{\bf Lemma}[section]
\newtheorem{theorem}{\bf Theorem}[section]

\newtheorem{corollary}{\bf Corollary}[section]

\newtheorem{remark}{Remark}[section]

\newtheorem{example}{Example}[section]

\newcounter{SE}
\newenvironment{SE}{\refstepcounter{SE}\equation}{\tag{SE}\endequation}

\title{Some generic fractal  properties of bounded\\ self-adjoint operators}
\author{M. Aloisio\thanks{Corresponding author. Email: moacir@ufam.edu.br}, S. L. Carvalho, and C. R. de Oliveira}
\date{June 2021}

\begin{document}

\maketitle

\begin{abstract} We study generic fractal properties of  bounded self-adjoint operators through  lower and  upper generalized fractal dimensions of their spectral measures. Two groups of results are presented. Firstly, it is shown that the set of vectors whose associated spectral measures have lower (upper) generalized fractal dimension equal to zero (one) for every $q>1$ ($0<q<1$) is either empty or generic. The second one gives sufficient conditions, for separable regular spaces of operators, for the presence of  generic extreme dimensional values; in this context, we have a new proof of the celebrated Wonderland Theorem.
\end{abstract}  

\renewcommand{\thetable}{\Alph{table}}

%%%%%%%%%%%%%%%%%%%%%%%%%%%%%%%%%%%%%%%%%%%%%%%%%%%%%%%%%%%%%%%%%%%%%%%%%%%%%%%%%%%%%%%%%%%%%%%%%%%%%%%%%%%%%%%%%%%%%%%%%%%%%%%%%%%%%%%%%%%%%%%%%%%%%%%%%%%%%%%%%%%%%%%%%%%%%%%%%%%%%%%%%--Introduction--%%%%%%%%%%%%%%%%%%%%%%%%%%%%%%%%%%%%%%%%%%%%%%%%%%%%%%%%%%%%%%%%%%%%%%%%%%%%%%%%%%%%%%%%%%%%%%%%%%%%%%%%%%%%%%%%%%%%%%%%%%%%%%%%%%%%%%%%%%%%%%%%%%%%%%%%%%%%%%%%%%%%%%%%%%%%%%%%%%%%%%%

\section{Introduction}

Spectral and dynamical properties of self-adjoint operators have a fundamental role in quantum mechanics, and there are many subtleties among them; for instance: 1)~any self-adjoint operator may be approximated (with respect to the Hilbert-Schmidt norm) by a pure point operator (this is the Weyl-von Neumann Theorem~\cite{vonNeumann,Weyl}); 2)~in some topological spaces of self-adjoint operators, the set of elements with purely singular continuous spectra is generic (the conclusion of the so-called Wonderland Theorem~\cite{SimonWonderland}); 3)~dense point spectrum imply a form of dynamical instability \cite{ACdeOLMP}; etc. Here, we present two new subtle properties related to generic dimensional properties of spectral measures, which are summarized in Theorems~\ref{stabtheorem} and~\ref{mainproposion}. For technical simplicity, we restrict ourselves to bounded self-adjoint operators~$T$ acting on the complex and separable Hilbert space~$\mathcal H$; we denote by $\mu_\psi^T$ the spectral measure  of~$T$ associated with the state $\psi\in\mathcal H$; for each Borel set $\Lambda \subset \mathbb{R}$, $P^T(\Lambda)$ represents the spectral resolution of $T$ over $\Lambda$; by $\mu$ we always mean a finite nonnegative Borel measure on~$\mathbb{R}$. Here, for every complete metric space~$X$, we say that $\mathcal{R} \subset X$ is residual if it contains a generic (i.e., a dense $G_\delta$) set in~$X$. 

The main results are described in Subsections~\ref{subsectDimHer} and~\ref{subsectGenFPO}, along with some examples and dynamical consequences. In Section~\ref{secpreliminarie} we recall some concepts and results regarding dimensional properties of nonnegative Borel measures. The proofs of Theorems~\ref{stabtheorem} and~\ref{mainproposion} are left to Section~\ref{mainsingularsection}. 

%%%%%%%%%%%%%%%%%%%%%%%%%%%%%%%%%%%%%%%%%%%%%%--A--dimensional--heritage--%%%%%%%%%%%%%%%%%%%%%%%%%%%%%%%%%%%%%%%%%%%%%%%%%%%%%%%%%%%%%%%%%%%%%%%%%%%%%%%%%%%%%%%%%%%%%%%%%%%%%%%%%%%%%%%%%%%%%%%%%%%%%%%%%%%%%%%%%%%%%%%%%%%%%%%%%%%%%%%%%%%%%%%%%%%%%%%%%%%%%%

\subsection{A dimensional heritage}\label{subsectDimHer}
Let $T$ be a bounded self-adjoint operator on $\mathcal{H}$, pick two vectors $\psi,\varphi\in\mathcal H$ and, for each $k\in\mathbb N$, set 
\[
\psi_k = \psi+\frac1k \varphi;
\]
although $\psi_k\to\psi$ as $k\to\infty$, it is not clear which properties of~$\psi$ and/or~$\psi_k$ are inherited  from~$\varphi$. E.g., if $\varphi$ belongs to the point subspace of~$T$, this property is clearly not preserved if~$\psi$ belongs to the continuous subspace; moreover, $\psi_k$ is a ``mixed vector.''Roughly, the first result in this work says that for each $k\in\mathbb{N}$, (some) values of the generalized fractal dimensions of~$\mu^T_{\psi_k}$ satisfy the same bounds as the values of~$\mu^T_{\varphi}$, being therefore,  held by a large set of spectral measures associated with $T$. Roughly, the idea is to show that $\mu^T_{\psi_k}$ inherits such dimensional properties from $\mu^T_{\varphi}$, so the
set of the vectors whose (some) values of the generalized fractal dimensions of the spectral measures satisfy the same bounds as the values of~$\mu^T_{\varphi}$ is dense in $\mathcal{H}$ (since $\psi$ is arbitrary in $\mathcal{H}$), and then combine this with suitable $G_\delta$ properties, proven in \cite{ACdeOLMP}, to show that they hold for generic sets.

Before we present a precise formulation of this result (see Subsection~\ref{subsectProofThmHerig} for its proof), we need a small preparation. For $q >0$, let  $D_\mu^-(q)$ and $D_\mu^+(q)$ (see Definition~\ref{GFDdefinition} ahead) denote the {\em lower} and the {\em upper generalized fractal dimensions} of~$\mu$, respectively; recall also that the functions $q\mapsto D_\mu^\mp(q)$ are nonincreasing and if~$\mu$ has bounded support, then $0\le D_\mu^-(q)\le D_\mu^+(q)\le 1$, for all~$q>0$ \cite{Barbaroux!997} (see also Proposition~\ref{GFDproposition} ahead). 

\begin{theorem}\label{stabtheorem} Let $T$ be a bounded self-adjoint operator on $\mathcal{H}$ and  $\alpha,\beta \geq 0$.  
\begin{enumerate}
\item Let $s >1$. If there exists $0 \neq \varphi \in \mathcal{H}$ such that $D_{\mu_{\varphi}^T}^-(s) \leq \beta$, then 
\[\Lambda^-(T,s,\beta) := \{\psi \in {\mathcal{H}} \mid D_{\mu_\psi^T}^-(s) \leq  \beta\}\]
is a dense $G_\delta$ set (i.e., a generic set) in $\mathcal{H}$. 
\item Let $0<q<1$. If there exists $0 \neq \varphi \in \mathcal{H}$ such that $D_{\mu_{\varphi}^T}^+(q) \geq \alpha$, then
\[\Lambda^+(T,q,\alpha) := \{\psi \in {\mathcal{H}} \mid D_{\mu_\psi^T}^+(q) \geq  \alpha\}\]
is a dense $G_\delta$ set in $\mathcal{H}$. 
\end{enumerate}
\end{theorem}  

\begin{example}[Rank-one perturbation of the almost-Mathieu operator]\label{Mathieu}{\rm  Write $\delta_1 = (\delta_{1,n})_{n\in\mathbb{Z}}$ and  let~$H$ be a rank-one perturbation of a quasi-periodic operator, acting on $\ell^2(\mathbb{Z})$, given by the law
\[(Hu)_n=(H_{\lambda,\alpha,\theta,\kappa}u)_n := u_{n+1} + u_{n-1} + \kappa  \cos(\pi \alpha n + \theta) +  \lambda \langle \cdot,\delta_1 \rangle \delta_1,\]
where $\lambda \in [0,1]$, $\alpha \in [0,2\pi)$, $\theta \in [0,2\pi)$ and $\kappa>2$.  It was shown in \cite{GKT} that there exists a dense $G_\delta$ set of irrational numbers~$\Omega\subset [0,2\pi)$ such that, for every $\alpha \in \Omega$, every $\theta, \lambda, \kappa>2$ and every $q \in  (0, 1)$, $D_{\mu_{\delta_1}^{H}}^+(q) = 1$. It follows from Theorem~\ref{stabtheorem} that, for each $\alpha \in \Omega$ and each $0<q<1$,
\[
\big\{ \psi \in \ell^2(\mathbb{Z}) \mid D_{\mu_\psi^{H}}^+(q) =1\big\}
\]
is a  generic set in $\ell^2(\mathbb{Z})$.  This example is particularly interesting because, for such parameter values, the spectrum of $H_{\lambda,\alpha,\theta,\kappa}$ is always purely singular (see \cite{GKT}), and generically with maximum value of the upper dimensions ($0<q<1$), a result intuitively associated with absolutely continuous spectrum.
}
\end{example}

\begin{example}[Continuous one-dimensional free Hamiltonian]\label{freexample}{\rm  Let $H_0:\mathcal{H}^2(\mathbb{R}) \subset {\rm L}^2(\mathbb{R}) \rightarrow {\rm L}^2(\mathbb{R})$ be given by the law $(H_0 \psi)(x) = -\psi''(x)$, and set $H:=H_0 P^{H_0}([0,1])$ (a bounded self-adjoint operator). For each ${\theta_n} = \frac{1}{2}- \frac{1}{n+2}$, let  $\psi_{n} \in {\rm L}^2(\mathbb{R})$ be such that its Fourier transform satisfies, for each $t>0$, 
\[\widehat{\psi_{n}}(t) = \chi_{[0,1]}(t)\, t^{-{\theta_n}}.\]
It turns out that (\cite{Oliveira}, Section~8.4.1)
\[{\rm d}\mu_{{\psi_{n}}}^{H}(x) = \frac12 \chi_{[0,1]}(x)\, x^{-(\theta_n +1/2)} {\rm d}x.\]
It is straightforward to check that, for each $n\in\mathbb{N}$ and each $s\geq 2$,
\[D_{\mu_{{\psi_{n}}}^{H}}^{\mp}(s) \leq 1-2\theta_n = \frac{2}{n+2}.\]
Therefore,  by Theorem \ref{stabtheorem}, for every $s\geq 2$,
\[\
\big \{ \psi \in {{\rm L}^2(\mathbb{R})} \mid D_{\mu_\psi^{H}}^-(s) =0\big\} = \bigcap_{n \geq 1} \Lambda^-(H,s,\theta_n)
\]
is a dense $G_\delta$ set in ${\rm L}^2(\mathbb{R})$. This example is interesting because~$H$  has purely absolutely continuous spectrum and, generically, with minimum values of such lower dimensions, a result intuitively associated with singular spectrum.
} 
\end{example}
  
\begin{remark}\label{mainremark1}{\rm To the best knowledge of the present authors, the result presented in Example \ref{freexample} leads to a phenomenon which has never been discussed: there exist an operator whose spectrum is purely absolutely continuous and a generic set of vectors whose time-average return probabilities decay with arbitrarily slow polynomial rates (for sequences of time $t_j \rightarrow \infty$). Namely, in this case, generically in $\psi \in {{\rm L}^2(\mathbb{R})}$ (by (\ref{prob}) just ahead), for every $k\ge1$,
\[\limsup_{t \to \infty}  \frac{t^{1/k}}{t} \int_0^t |\langle \psi, e^{-iHt}\psi \rangle|^2 {\rm d}s = \infty.\]
We note that this is, in some sense, the counterpart of the following situation: an operator with pure point spectrum and a generic set of states whose spectral measures have maximal upper generalized dimension (such is the case of the operator discussed in Example~\ref{Mathieu}); such states are, therefore, delocalized (see Subsection~\ref{subsecdyn} for details and~\cite{ACdeOLMP,GKT}).}
\end{remark}

Next we turn to the second group of generic results in this work.

%%%%%%%%%%%%%%%%%%%%%%%%%%%%%%%%%%%%%%%%%%%%%%%%%%%%%%--Generic--fractal--properties--%%%%%%%%%%%%%%%%%%%%%%%%%%%%%%%%%%%%%%%%%%%%%%%%%%%%%%%%%%%%%%%%%%%%%%%%%%%%%%%%%%%%%%%%%%%%%%%%%%%%%%%%%%%%%%%%%%%%%%%%%%%%%%%%%%%%%%%%%%%%%%%%%%%%%%%%%%%%%%%%%%%%%%%%%%%%%%%%%%%%%%%%%%%%%%

\subsection{Generic fractal properties of spectral measures}\label{subsectGenFPO}

Recall that  a metric space $(X,d)$ of self-adjoint operators acting in~$\mathcal{H}$ is {\em regular}~\cite{SimonWonderland} if it is complete and convergence in the metric~$d$ implies strong resolvent convergence.  Denote by $C_{\mathrm{p}}=C_{\mathrm{p}}(X)$ the set of operators $T\in X$ with pure point spectrum and by  $C_{\mathrm{ac}}=C_{\mathrm{ac}}(X)$  the set of operators $T\in X$ with purely absolutely continuous spectrum.  

Under some assumptions, a version of the Wonderland Theorem related to extreme correlation dimensional values (i.e., $D_{\mu_\psi^{T}}^-(2)=0 $ and $D_{\mu_\psi^{T}}^+(2)=1$) of spectral measures was proven in~\cite{CarvalhoCorrelation}, and dynamical consequences were explored. In the following, we extend this result to all dimensions $D_{\mu_\psi^{T}}^\mp(q)$, $q>0$, for separable regular sets of bounded self-adjoint operators. These fine dimensional properties will also imply generic singular continuous spectrum; such  results are gathered in the next  statements. As a spinoff, for such spaces we have a new proof of the Wonderland Theorem (see Corollary~\ref{FWT}). 

\begin{theorem}\label{mainproposion} Let~$X$ be a separable regular space of bounded self-adjoint operators. If both sets $C_{\mathrm{p}}$ and $C_{\mathrm{ac}}$ are dense in~$X$, then there exists a generic set $\mathcal{M}$  in $\mathcal{H}$ such that, for each $\psi \in \mathcal{M}$, the set  $X_{{01}}(\psi) := \{T\in X \mid D_{\mu_{\psi}^T}^-(q) =0$ and $D_{\mu_{\psi}^T}^+(q) =1$, for all  $q >0\}$ is residual in~$X$.
\end{theorem}

\begin{remark}\label{anularemark}{\rm  We  note that even for $T$ with pure point spectrum, it may occur that $D_{\mu_{\psi}^T}^-(q) > 0$ for all $0 < q < 1$ (see \cite{BGT2001,GKT} for details), and when $T$ has purely absolutely continuous spectrum, it may happen that $D_{\mu_{\psi}^T}^+(q) < 1$ for all $q > 1$ (see Example~\ref{freexample}). Therefore, such result is not necessarily expected.}
\end{remark}

\begin{corollary}[Wonderland Theorem]\label{FWT} Let~$X$ be as in Theorem~\ref{mainproposion}.  If both sets $C_{\mathrm{p}}$ and  $C_{\mathrm{ac}}$  are dense in~$X$, then the set  $C_{\mathrm{sc}}=C_{\mathrm{sc}}(X):= \{T\in X\mid T$ has purely singular continuous spectrum$\}$ is residual in~$X$.
\end{corollary}

\begin{remark}\label{remarkWTwithDims}{\rm 
(a)~The proof of Corollary~\ref{FWT} presented below is entirely based on the conclusions of Theorem~\ref{mainproposion}, that is, it is based on the existence of the residual sets~$\mathcal{M}$ and  $X_{{01}}(\psi)$, for $\psi\in\mathcal M$. It is, therefore, a  different proof from the one presented in~\cite{SimonWonderland}. (b)~Naturally, one may combine Theorem~\ref{mainproposion} and Corollary~\ref{FWT} to conclude that for each $\psi \in\mathcal M$, the set $X_{{01}}^{\mathrm{sc}}(\psi) := \{T\in X\mid T$ has purely singular continuous spectrum, $D_{\mu_{\psi}^T}^-(q) =0$ and $D_{\mu_{\psi}^T}^+(q) =1$, for all  $q>0\}$ is residual in~$X$. Indeed, it is enough to note that $X_{{01}}^{\mathrm{sc}}(\psi)=  X_{{01}}(\psi)\cap C_{\mathrm{sc}}$.}
\end{remark}

\begin{remark}\label{regularremark}{\rm Since for bounded self-adjoint operators on~$\mathcal H$ strong convergence implies strong resolvent convergence~\cite{Oliveira}, every metric space of bounded self-adjoint operators endowed with the strong operator topology is a  regular space, and by~\cite{Chan} it is also separable.}
\end{remark}

Let $\dim_{\mathrm H}^+(\mu)$ denote the upper  Hausdorff dimension of~$\mu$ (such notion is recalled in Section~\ref{secpreliminarie}). The next result, presented in \cite{Barbaroux!997}, relates this quantity to the lower generalized fractal dimensions.

\begin{proposition}\label{HPandGFDproposition} Let~$\mu$ be a finite nonnegative Borel measure on~$\mathbb{R}$ and  $0 < q < 1 < s$. Then, $D_{\mu}^-(q) \geq {\rm dim_{H}^+}(\mu) \geq  D_{\mu}^-(s)$.
\end{proposition}

\begin{lemma}\label{sclemma} Let~$T$ be a bounded self-adjoint operator on~$\mathcal{H}$ and $0 \neq \psi\in\mathcal{H}$. If there exist $0 < q' < 1 < s'$ such that $ D_{\mu_{\psi}^T}^-(q') < 1$ and  $D_{\mu_{\psi}^T}^+(s')  > 0$, then $\mu_\psi^T$ is a purely singular continuous measure. 
\end{lemma}

\begin{proof} If $\mu_{\psi}^T$ has an atom, that is, if there exists $\lambda \in \mathbb{R}$ such that $\mu_{\psi}^T(\{\lambda\})>0$, then it is easy to show that for each $s>1$, $D_{\mu_{\psi}^T}^+(s)  = 0$ (see (\ref{GFD0}) ahead). On the other hand, if $\mu_{\psi}^T$ has an absolutely continuous component, then ${\rm dim_{H}^+}(\mu_{\psi}^T) = 1$ and, therefore, it follows from Proposition \ref{HPandGFDproposition} that for each $0<q<1$, $D_{\mu_{\psi}^T}^-(q) = 1$. Hence, if there exist $0<q'<1<s'$  with  $D_{\mu_{\psi}^T}^+(s') > 0$ and $D_{\mu_{\psi}^T}^-(q')< 1$, then $\mu_{\psi}^T$ is singular continuous. 
\end{proof}

\begin{proof}(Corollary~\ref{FWT}) {\rm Let $\mathcal{M}$  be as in the statement of Theorem~\ref{mainproposion} and let $\{\psi_j\}_{j \in \mathbb{Z}} \subset \mathcal{M}$ be a dense sequence in $\mathcal{H}$. It follows from Lemma~\ref{sclemma} that for each $\phi \in \mathcal{M}$ and each $S\in X_{{01}}(\phi)$, the spectral measure $\mu_\phi^S$ is purely singular continuous. Then, since the singular continuous subspace associated with each self-adjoint operator is a closed subspace of~$\mathcal H$ \cite{Oliveira}, one has  
\[
C_\mathrm{sc}\, \supset \, \cap_{j \in \mathbb{Z}} X_{{01}}(\psi_j).
\] The result is now a consequence of Theorem~\ref{mainproposion}.}
\end{proof}

The following result is a direct consequence of  Remark~\ref{remarkWTwithDims} (b) and Proposition~\ref{HPandGFDproposition}. 

\begin{corollary}\label{maincorollary}Let~$X$ be as in Theorem~\ref{mainproposion}. If both sets $C_{\mathrm{p}}$ and  $C_{\mathrm{ac}}$  are dense in~$X$,  then there exists a generic set $\mathcal{M}$  in $\mathcal{H}$ such that, for each $\psi \in \mathcal{M}$, the set $\{T\in X\mid T$ has purely singular continuous spectrum, $\dim_{\rm H}^+(\mu_{\psi}^T) =0 \}$ is residual in~$X$.
\end{corollary} 

Corollary~\ref{maincorollary}  is also a consequence of the results recently presented in~\cite{CarvalhodeOliveiraHP}. However, the results and methods of this paper are different from those of~\cite{CarvalhodeOliveiraHP}.  Namely, the main technical ingredients in the present paper involve some decompositions of spectral measures with respect to the fractal generalized dimensions, whereas in~\cite{CarvalhodeOliveiraHP} the main idea is to directly show that for each $0\neq\psi\in\mathcal{H}$, $\{T\in X\mid \dim_{\rm H}^+(\mu_{\psi}^T) =0\}$ is a $G_\delta$ set in~$X$. 

There are in the literature (see, for instance, \cite{CarvalhoCorrelation,CarvalhodeOliveiraPAMS,DamanikDDP,Oliveira,SimonWonderland}) numerous important examples for which our general results  apply. As an illustration, we present the following application. 

\begin{example}\label{exampAnQuas}{\rm Consider the class of Schr\"odinger operators with analytic quasiperiodic potentials, acting on $\ell^2(\mathbb{Z})$,  generated by a nonconstant real analytic function $v \in {\mathrm{C}}^\omega(\mathbb{T}, \mathbb{R})$, that is, 
\[(H_{\lambda,\alpha,\theta}^vu)_n := u_{n+1} + u_{n-1} + \lambda v(\theta + \alpha n)u_n,\]
where $0 \neq \lambda \in \mathbb{R}$, $\alpha \in \mathbb{T}$ is the frequency and $\theta \in \mathbb{T}$ is the phase (an important example is given by the almost Mathieu operator, for which $v(x) = 2 \cos(2 \pi x)$; see Example~\ref{Mathieu}).

For each nonconstant $v \in {\mathrm{C}}^\omega(\mathbb{T}, \mathbb{R})$,  $0 \neq \lambda \in \mathbb{R}$ and $\theta \in \mathbb{T}$,  consider the space of self-adjoint operators
\[X_{\lambda,\theta}^v : = \{ H_{\lambda,\alpha,\theta}^v \mid \,  \alpha \in \mathbb{T}\}\]
endowed with the following metric  (whose induced topology is equivalent to the strong operator topology) 
\[d(H_{\lambda,\alpha,\theta}^v,H_{\lambda,\alpha',\theta}^v) := \biggr|\sin\biggr(\frac{\alpha-\alpha'}{2}\biggr)\biggr|.\]
Since for  $\lambda \in \mathbb{R}$,  $\theta \in \mathbb{T}$ and  $\alpha \in \mathbb{Q}/\mathbb{Z}$, the operator $H_{\lambda,\alpha,\theta}^v$ is purely absolutely continuous, and there exists $\lambda_0(v) > 0$  \cite{Bourgain,Bourgain2} (for the almost Mathieu operator, one can take $\lambda_0 = 1$)  so that, for every $\lambda > \lambda_0(v) > 0$, every $\theta \in \mathbb{T}$ and for all $\alpha$ outside a set of zero Lebesgue measure, $H_{\lambda,\alpha,\theta}^v$ is pure point, it follows from Theorem~\ref{mainproposion} (see also Remark~\ref{remarkWTwithDims}) that there exists a generic set $\mathcal{M}(v)$ in $\ell^2(\mathbb{Z})$ such that, for each $\psi \in \mathcal{M}(v)$, the set $\{ H = H_{\lambda,\alpha,\theta}^v \in X_{\lambda,\theta}^v\mid H$ has purely singular continuous spectrum with $D_{\mu_{\psi}^{H}}^-(q) =0$ and $D_{\mu_{\psi}^{H}}^+(q) =1$, for all $q >0\}$ is residual in~$X_{\lambda,\theta}^v$.} 
\end{example}

\begin{remark}{\rm We note that under the above assumptions, Bourgain and Goldstein has shown in \cite{Bourgain2} that for every $\alpha$ outside a set of zero Lebesgue measure, $H_{\lambda,\alpha,\theta}^v$ has dynamical localization; therefore, the conclusions regarding the lower dimensions in Example~\ref{exampAnQuas} follow from Theorem~4.3 in \cite{CarvalhodeOliveiraPAMS}.}
\end{remark}

%%%%%%%%%%%%%%%%%%%%%%%%%%%%%%%%%%%%%%%%%%%%%%%%%%%%%%%%--Dynamical--consequences--%%%%%%%%%%%%%%%%%%%%%%%%%%%%%%%%%%%%%%%%%%%%%%%%%%%%%%%%%%%%%%%%%%%%%%%%%%%%%%%%%%%%%%%%%%%%%%%%%%%%%%%%%%%%%%%%%%%%%%%%%%%%%%%%%%%%%%%%%%%%%%%%%%%%%%%%%%%%%%%%%%%%%%%%%%%%%%%%%%%%%%%%%%%%%%%%%

\subsubsection{Remarks on dynamical consequences}\label{subsecdyn}

Here, we explore some dynamical consequences of the general above results. We recall that if~$T$ is a bounded self-adjoint operator acting on~$\mathcal H$,  then ${\mathbb{R}} \ni t \mapsto e^{-itT}$ is a one-parameter strongly continuous unitary evolution group and, for each $\psi\in\mathcal{H}$, $(e^{-itT}\psi)_{t \in \mathbb{R}}$ is the unique wave packet solution to the Schr\"odinger equation
\begin{SE}\label{SE}
\begin{cases} \partial_t \psi = -iT\psi, \quad t \in {\mathbb{R}}, \\ \psi(0) = \psi\in\mathcal{H}.  \end{cases} \end{SE}

Next, we list two quantities usually considered to probe the large time behaviour of the dynamics $e^{-itT}\psi$.  The first one is the so-called (time-average) {\em quantum return probability} $\langle \gamma_\psi^T \rangle (t)$, which gives the (time-average) probability of finding the particle at time $t>0$ in its
initial state $\psi$:
\begin{equation*}
\langle \gamma_\psi^T \rangle (t) :=  \frac{1}{t}\int_0^t |\langle \psi, e^{-isT} \psi \rangle|^2 \, {\mathrm d}s;         
\end{equation*}
its lower and upper decaying exponents are given, respectively, by \cite{Barbaroux1,holsch} 
\begin{equation}\label{prob}
\liminf_{t \to \infty} \frac{\ln \langle \gamma_\psi^T \rangle (t) }{\ln t} = -D_{\mu_\psi^T}^+(2), \qquad 
\limsup_{t \to \infty}  \frac{\ln  \langle \gamma_\psi^T \rangle (t)}{\ln t} = -D_{\mu_\psi^T}^-(2).    
\end{equation}

In order to probe dynamical (de)localization associated with an initial state~$\psi$ with respect to a general orthonormal basis~$\mathfrak B =\{\eta_j\}$ of~$\mathcal H$, one may quantify the ``travel to large dimensions~$j$'' by considering the time evolution of the (time-average) {\em $p$-moments} of~$\psi$, $p>0$, that is,
\[r_{p,\mathfrak B}^{\psi,T}(t):=\biggl( \frac1t \int_0^t \sum_{j} |n|^p |\langle \eta_j,e^{-isT}\psi\rangle|^2\;\mathrm{d}s\biggl)^{\frac{1}{p}}. \]
If one thinks of a polynomial growth $r_{p,\mathfrak B}^{\psi,T}(t)\sim t^{\beta(p)}$, then the {\em lower} and {\em upper $p$-moment growth exponents} are then naturally introduced, respectively, by
\begin{equation*}\beta_{\psi,T}^-(p,\mathfrak B):= \liminf_{t\to\infty}\, \frac{\ln r_{p,\mathfrak B}^{\psi,T}(t)}{\ln t},\qquad \beta_{\psi,T}^+(p,\mathfrak B):= \limsup_{t\to\infty}\, \frac{\ln r_{p,\mathfrak B}^{\psi,T}(t)}{\ln t}.
\end{equation*}

The following inequality,  due to  Barbaroux, Germinet and Tcheremchantsev~\cite{BGT2001}, and independently obtained by Guarneri and Schultz-Baldes~\cite{GuarSB1999b}, 
\begin{equation}\label{BGTinequality}\beta_{\psi,T}^{\mp}(p,\mathfrak B) \geq    D_{\displaystyle\mu_{\psi}^T}^{\mp}\biggr(\frac{1}{1+p}\biggr),
\end{equation}
holds for all orthonormal bases and all $p>0$. Such notions are particularly interesting when $T$ is a Schr\"odinger operator acting in $\ell^2(\mathbb{Z}^\nu)$, $\nu \in \mathbb{N}$, $\{\eta_j\}$ is a basis of $\ell^2(\mathbb{Z}^\nu)$ and $\psi = f(T)\eta_0$, with $f \in C_0^\infty({\mathbb{R}})$, so displaying some locality condition \cite{DamanikDDP2,Germinetmoments}.

The corresponding dynamical consequences of Theorem~\ref{mainproposion} come from (\ref{prob}) and  (\ref{BGTinequality}); namely, the typical dynamical situation is characterized by the fact that the decay rates of the {\em quantum return probability} assume their extreme values, and by occurrence of {\em weak dynamical delocalization}:
for a typical $T \in X$, for every $\psi \in \mathcal{M}$ and for all $p>0$, $\beta_{\psi,T}^+(p,\mathfrak B) \geq 1$ and so,  
\[r_{p,\mathfrak B}^{\psi,T}(t_j)\sim\; t_j^{\beta_{\psi}^{+}(p,\mathfrak B)}\]
for a sequence of instants of time $t_j\to\infty$.  The term {\em weak} is due to the possibility of $\beta_{\psi,T}^-(p,\mathfrak B)=0$, for all~$p>0$. 

\begin{remark}{\rm Given  $T \in X_{{01}}(\psi)$ and a fixed orthonormal basis ${\mathfrak B}$ of $\mathcal{H}$, 
\[G({\mathfrak B}) = \{\phi \mid \, r_{p,{\mathfrak {B}}}^{\phi,T}(t) \equiv \infty \, \, \text{for  all} \, \, p>0 \} \]
is a dense $G_\delta$  set in $\mathcal{H}$ (see Proposition A.2 in \cite{ACdeOLMP}). In this case, ${\mathcal M} \cap G(\mathfrak B)$ is always a dense $G_\delta$  set in $\mathcal{H}$ as well. To get finite moments, usually it is necessary some  ``energy localization condition'' on the vector (with respect to the basis $\mathfrak B$); see Section 2.7 in  \cite{DamanikDDP2}. For instance, for each $\psi \in {\mathcal M}$, if $T \in X_{{01}}(\psi)$ is a discrete Schr\"odinger operator, for each orthonormal basis ${\mathfrak B} = \{\eta_j\}$ such that $\eta_1 = \Vert\psi\Vert^{-1}\psi$, one has $\beta_{\psi}^+(p,{\mathfrak B}) \leq 1$ for all $p>0$; see \cite{Germinetmoments,DamanikDDP2} for details. Hence,  note that ${\mathcal M} \setminus G({\mathfrak B}) \not= \emptyset$.}
\end{remark}

\begin{remark}\label{mainremark2}{\rm We note that a natural strategy to prove Theorem~\ref{mainproposion} consists in showing that there exist a dense subset of $T \in X$ such that $\mu_\psi^T$ is 1-H\"older continuous (since in this case, $D_{\mu_{\psi}^T}^-(q) =1$ for all $q>0$), and a dense subset of operators in~$X$ with dynamical localization, that is, satisfying for each $\psi\in\mathcal{H}$ and each $p>0$, $\displaystyle\limsup_{t \to \infty} r_{p,\mathfrak B}^{\psi,T}(t) < +\infty$ (since in this case, by (\ref{BGTinequality}), $D_{\mu_{\psi}^T}^+(q) =0$ for all $q >0$). However, it is well known that there are particular families of Schr\"odinger operators with absolutely continuous spectrum and vectors $\psi \in \ell^2(\mathbb{Z}^\nu)$ satisfying some locality condition (for instance, $\psi \in \ell^1(\mathbb{Z}^\nu)$) such that $\mu_\psi^T$ is at most $1/2$-H\"older continuous (see \cite{AvilaUaH,DamanikIDS} for additional comments). On the other hand, in order to show dynamical localization (for instance, for Schr\"odinger operators), the usual techniques involve the notion known as SULE \cite{deOliveiraPigossi,delRioSC2,StolzIntr}, which also implies spectral localization \cite{delRioSC2} (we note that localization also depends on some locality condition on $\psi$ \cite{DamanikDDP2,Germinetmoments}). Thus, although natural, this strategy does not seem to be suitable for the rather general setting of this work.}
\end{remark}

\begin{remark}{\rm {It is worth underlying that by combining some results of \cite{CarvalhoCorrelation,CarvalhodeOliveiraPAMS}, one gets Theorem~\ref{mainproposion} for the particular space of one-dimensional Jacobi matrices (with a necessarily nontrivial restriction of the spectrum to obtain $1$-H\"older continuity) endowed with the topology of pointwise convergence. However, the strategy followed in  \cite{CarvalhoCorrelation,CarvalhodeOliveiraPAMS} does not seem to be adequate to prove this theorem in such generality.}}
\end{remark}

%%%%%%%%%%%%%%%%%%%%%%%%%%%%%%%%%%%%%%%%%%%%%%%%%--Dimensions--of--measures:--a--short--account--%%%%%%%%%%%%%%%%%%%%%%%%%%%%%%%%%%%%%%%%%%%%%%%%%%%%%%%%%%%%%%%%%%%%%%%%%%%%%%%%%%%%%%%%%%%%%%%%%%%%%%%%%%%%%%%%%%%%%%%%%%%%%%%%%%%%%%%%%%%%%%%%%%%%%%%%%%%%%%%

\section{Dimensions of measures: a short account}
\label{secpreliminarie} 

%%%%%%%%%%%%%%%%%%%%%%%%%%%%%%%%%%%%%%%%%%%%%%%%%--Hausdorff--and--packing--decompositions--%%%%%%%%%%%%%%%%%%%%%%%%%%%%%%%%%%%%%%%%%%%%%%%%%%%%%%%%%%%%%%%%%%%%%%%%%%%%%%%%%%%%%%%%%%%%%%%%%%%%%%%%%%%%%%%%%%%%%%%%%%%%%%%%%%%%%%%%%%%%%%%%%%%%%%%%%%%%%%%%%%%

\subsection{Hausdorff dimension}\label{HPM1}

\DEFI\label{locexpdefinition} Let $\mu$ be a finite nonnegative Borel measure on $\mathbb{R}$. The pointwise lower scaling exponent of~$\mu$  at $x \in \mathbb{R}$ is defined as  
\[ d_\mu^-(x) := \liminf_{\epsilon \downarrow 0} \frac{\ln \mu (B(x,\epsilon))}{\ln \epsilon}\]
if, for all $\epsilon>0$,  $\mu(B(x;\epsilon))> 0$; otherwise, one sets $d_\mu^{-}(x) := \infty$.
\DEFF

\begin{definition}\label{defHPD}{\rm The upper Hausdorff dimension of~$\mu$  is defined as}
\[{\rm dim}^+_{{\rm H}} (\mu)  := \mu{\rm {\text -}}\esssup d_\mu^-.\]
\end{definition}

The next result presents a dimensional property of measures that are absolutely continuous with respect to the Lebesgue measure (see  \cite{Falconer,Guarneri1} for details).

\begin{proposition}\label{scexpproposition0} Let $\mu$ be a nonnegative Borel measure on $\mathbb{R}$ which is absolutely continuous with respect to the Lebesgue measure. Then, $\mu$-$\essinf d_\mu^- =1$. 
\end{proposition} 

%%%%%%%%%%%%%%%%%%%%%%%%%%%%%%%%%%%%%%%%%%%%%%%%%%%%%%%%%%%--Generalized--dimensions--%%%%%%%%%%%%%%%%%%%%%%%%%%%%%%%%%%%%%%%%%%%%%%%%%%%%%%%%%%%%%%%%%%%%%%%%%%%%%%%%%%%%%%%%%%%%%%%%%%%%%%%%%%%%%%%%%%%%%%%%%%%%%%%%%%%%%%%%%%%%%%%%%%%%%%%%%%%%%%%%%%%%%%%%%%

\subsection{Generalized dimensions}

The study of fractal dimensions of spectral measures in the context of quantum mechanics appeared as an attempt to answer the following question (posed in~\cite{Kraut}): ``What determines the spreading of a wave packet?" For a broader discussion, we highlight also the works \cite{BGT2001,Combes,GKT,Guarneri1,GuarSB1999b,Last}.

\begin{definition}\label{GFDdefinition}{\rm Let $\mu$ be a finite positive Borel measure on $\mathbb{R}$. The lower and upper $q$-generalized fractal dimensions, $q>0$, $q \neq 1$, of~$\mu$  are defined, respectively, as  
\[D_\mu^-(q) := \liminf_{\epsilon \downarrow 0} \frac{\ln \biggr[\int \mu (B(x,\epsilon))^{q-1} {\mathrm d}\mu(x)\biggr]}{(q-1) \ln \epsilon} \quad{\rm  and }\quad D_\mu^+(q) := \limsup_{\epsilon \downarrow 0} \frac{\ln \biggr[\int \mu (B(x,\epsilon))^{q-1} {\mathrm d}\mu(x)\biggr]}{(q-1)\ln \epsilon},
\]
with the integrals taken over the support of $\mu$. The lower and upper $1$-generalized fractal dimensions of~$\mu$  are defined, respectively, as  
\[D_\mu^-(1) := \liminf_{\epsilon \downarrow 0} \frac{\int\ln \mu (B(x,\epsilon)) {\mathrm d}\mu(x)}{\ln \epsilon} \quad{\rm  and }\quad D_\mu^+(1) := \limsup_{\epsilon \downarrow 0} \frac{\int\ln \mu (B(x,\epsilon)) {\mathrm d}\mu(x)}{\ln \epsilon};\]
again, the integrals are taken over the support of $\mu$.}
\end{definition}

Other important quantities related to the $q$-generalized fractal dimensions, $q>0$, $q \neq 1$, are the so-called mean $q$-dimensions. 

\begin{definition}\label{defMeanDim}{\rm Let $\mu$ be a finite positive Borel measure on $\mathbb{R}$ and~$q>0$. The lower and upper mean $q$-dimensions of~$\mu$ are defined, respectively, as}  
\[m_\mu^-(q) := \liminf_{\epsilon \downarrow 0} \frac{\ln [\epsilon^{-1} \int \mu (B(x,\epsilon))^q \, {\mathrm d}x]}{(q-1) \ln \epsilon} \quad {\rm  and } \quad m_\mu^+(q) := \limsup_{\epsilon \downarrow 0} \frac{\ln [\epsilon^{-1} \int \mu (B(x,\epsilon))^q \, {\mathrm d}x]}{(q-1)\ln \epsilon}.\]
\end{definition}

\begin{proposition}[Theorem 2.1 and Propositions 3.1 and 3.3 in \cite{Barbaroux!997}]\label{GFDproposition} Let $\mu$ be as before. Then,
\begin{enumerate} 
\item For every $q>0$, $q \not = 1$, $D_\mu^{\mp}(q) = m_\mu^{\mp}(q)$.
\item $D_\mu^-(q)$ and $D_\mu^+(q)$ are nonincreasing functions of $q >0$.
\item If~$\mu$ has bounded support, then for all $q>0$, $0 \leq D_\mu^-(q) \leq D_\mu^+(q) \leq 1$.
\end{enumerate}
\end{proposition}

The next results play a fundamental role in the proof of Theorem~\ref{mainproposion}. 

\begin{proposition}[Proposition 3.1 in~\cite{ACdeOLMP}] \label{propgdelta}  Let~$T$ be a bounded self-adjoint operator on~$\mathcal{H}$ and   $q > 0$, $q \neq 1$. Then, for every $\Gamma \geq 0$,
\begin{enumerate}
\item $\{\psi \in {\mathcal{H}} \mid  D_{\mu_\psi^T}^-(q) \leq  \Gamma \}$ is a $G_\delta$ set in $\mathcal{H}$,
\item $\{\psi \in {\mathcal{H}} \mid D_{\mu_\psi^T}^+(q) \geq  \Gamma \}$ is a $G_\delta$ set in $\mathcal{H}$. 
\end{enumerate}
\end{proposition}

\begin{proposition}\label{propgdeltabis} Let $(X,d)$ be a regular space of bounded self-adjoint operators and let  $q > 0$, $q \neq 1$. Then, for every  $0 \neq \psi \in \mathcal{H}$ and every $\Gamma \geq 0$,
\begin{enumerate}
\item $\{T \mid  D_{\mu_\psi^T}^-(q) \leq  \Gamma \}$ is a $G_\delta$ set in~$X$,
\item $\{T \mid D_{\mu_\psi^T}^+(q) \geq  \Gamma \}$ is a $G_\delta$ set in~$X$. 
\end{enumerate}
\end{proposition}

Since the proof of Proposition~\ref{propgdeltabis} is based on the same arguments presented in the proof of Proposition~\ref{propgdelta} (discussed in details in~\cite{ACdeOLMP}), it will be omitted.

%%%%%%%%%%%%%%%%%%%%%%%%%%%%%%%%%%%%%%%%%%%%%%%%%%%%%%%%%%%%%%%%%%%%%%%%%%%%%%%%%%%%%%%%%%%%%%%%%%%%%%%%%%%%%%%%%%%%%%%%%%%%%%%%%%%%%%%%%%%%%%%%%%%%%%%%%%%%%%%%%%%--Fractal--s.c.--spectrum--generic--%%%%%%%%%%%%%%%%%%%%%%%%%%%%%%%%%%%%%%%%%%%%%%%%%%%%%%%%%%%%%%%%%%%%%%%%%%%%%%%%%%%%%%%%%%%%%%%%%%%%%%%%%%%%%%%%%%%%%%%%%%%%%%%%%%%%%%%%%%%%%%%%%%%%%%%%%%%%%%%%%%%%%%%%%%%%%%%%%%%%%%%%%

\section{Proofs}\label{mainsingularsection}

%%%%%%%%%%%%%%%%%%%%%%%%%%%%%%%%%%%%%%%%%%%%%%%%%%%%%%--Proof--of--Proposition--1.1--%%%%%%%%%%%%%%%%%%%%%%%%%%%%%%%%%%%%%%%%%%%%%%%%%%%%%%%%%%%%%%%%%%%%%%%%%%%%%%%%%%%%%%%%%%%%%%%%%%%%%%%%%%%%%%%%%%%%%%%%%%%%%%%%%%%%%%%%%%%%%%%%%%%%%%%%%%%%%%%%%%%%%%%%%%%

\subsection{Proof of Theorem \ref{stabtheorem}}\label{subsectProofThmHerig}

\noindent {\em 1.} 
\noindent {\bf Case 1.} $T$ has an eigenvalue. Then, there exist $\lambda \in \mathbb{R}$ and $0 \neq \eta \in \mathcal{H}$ such that $T\eta = \lambda \eta$. Set, for each $\psi \in \mathcal{H}$ and each $k\in\mathbb{N}$, $\psi_k:=\psi + \frac{1}{k}\eta$ and note that $\displaystyle\lim_{k \to \infty}\Vert\psi_k-\psi\Vert=0$; for all sufficiently large $k$, $P^T_{\mathrm p}\psi_k \neq 0$, where $P^T_{\mathrm p}$ is the orthogonal projection onto the pure point subspace of~$T$ (namely, this is true for each $k\in\mathbb{N}$ if $P^T_{\mathrm p}\psi=0$; otherwise, there exists $k_0$ such that for each $k\ge k_0$, $\Vert P^T_{\mathrm p}\psi-\frac{1}{k}\eta\Vert>0$). Therefore, since $\{\psi \in {\mathcal{H}}\mid \Vert P^T_{\mathrm p} \psi \Vert > 0\}$ is open and dense, the result follows from the set inclusions  
\[\{\psi \in {\mathcal{H}} \mid D_{\mu_\psi^T}^-(s) \leq  \beta\} \supset \{\psi \in {\mathcal{H}} \mid D_{\mu_\psi^T}^-(s) =0\} \supset \{\psi \in {\mathcal{H}} \mid \Vert P_{\mathrm p} \psi \Vert > 0\}.\]

Namely, if $\xi:=P^T_{\mathrm{p}}\psi\neq 0$, then $\mu_\xi^T$ has an atom, i.e, there exists $\zeta\in\mathbb{R}$ such that $\mu_{\xi}^T(\{\zeta\})>0$;  so, since for each $s>1$ and each $0<\epsilon <1$, $(s-1)\ln \epsilon <0$, one has 
\begin{equation}\label{GFD0}
D_{\mu_{\xi}^T}^-(s) \leq D_{\mu_{\xi}^T}^+(s)  = \limsup_{\epsilon \to 0} \frac{\ln\left[ \int \mu_{\xi}^T(B(x,\epsilon))^{s-1} \mathrm{d}\mu_{\xi}^T(x) \right]}{(s-1)\ln \epsilon} \leq \limsup_{\epsilon \to 0} \frac{\ln\left[\mu_{\xi}^T(\{\zeta\})^s \right]}{(s-1)\ln \epsilon} =0.     
\end{equation}

\

\noindent {{\bf Case 2.}} The spectrum of $T$ is purely continuous. Let $0 \neq \varphi \in \mathcal{H}$ be such that $D_{\mu_{\varphi}^T}^-(s) \leq \beta$.  Firstly, let us build a sequence of decreasing compact sets $(A_k)$ so that, for each $k \geq 1$, $D_{\mu^k}^-(s) \leq \beta$, where, for every Borel set $\Lambda \subset \mathbb{R}$, $\mu^k(\Lambda) := \mu_{\varphi}^T(\Lambda \cap A_k)$. 

Let $r>0$ be such that $\supp(\mu_\varphi^T) \subset [-r,r]$;  set $I_1 := [-r,0]$ and $I_2 := [0,r]$. Since $D_{\mu_{\varphi}^T}^-(s) \leq \beta$, one has, for every $\sigma>0$,   
\begin{eqnarray*}
\infty &=& \limsup_{\epsilon \to 0}  \epsilon^{(\sigma+\beta)(1-s)}   \int_{-r}^r \mu_\varphi^T(B(x,\epsilon))^{s-1} {\rm d}\mu_{\varphi}^T(x)\\ &\leq&  \sum_{j=1}^2\limsup_{\epsilon \to 0} \epsilon^{(\sigma+\beta)(1-s)}  \int_{I_j} \mu_\varphi^T(B(x,\epsilon))^{s-1} {\rm d}\mu_{\varphi}^T(x),
\end{eqnarray*}
and so, there exists $j_1 \in \{1,2\}$ such that 
\[\limsup_{\epsilon \to 0} \epsilon^{(\sigma+\beta)(1-s)}  \int_{I_{j_1}} \mu_\varphi^T(B(x,\epsilon))^{s-1} {\rm d}\mu_{\varphi}^T(x) = \infty.\]

Now, write $I_{j_1} = [a_1,b_1]$ and define  $A_1 := L_1 \cup I_{j_1} \cup L_1'$, where $L_1 = [-|I_{j_1}|/4+a_1,a_1]$ and $L_1' = [b_1,b_1+|I_{j_1}|/4]$. Set $\mu^1(\cdot) := \mu_{\varphi}^T(\cdot\cap A_1)$. Then, for every $0<\epsilon < |I_{j_1}|/4$,
\begin{eqnarray*}
\int_{A_1} \mu^1(B(x,\epsilon))^{s-1} {\rm d}\mu^1(x)  &=& \int_{A_1} \mu_{\varphi}^T(A_1 \cap B(x,\epsilon))^{s-1} {\rm d}\mu_{\varphi}^T(x) \\ &\geq&  \int_{I_{j_1}} \mu_{\varphi}^T(A_1 \cap B(x,\epsilon))^{s-1}{\rm d}\mu_{\varphi}^T(x)\\ &=& \int_{I_{j_1}} \mu_{\varphi}^T(B(x,\epsilon))^{s-1}{\rm d}\mu_{\varphi}^T(x).
\end{eqnarray*}
Thus, for every $\sigma>0$,
\[\limsup_{\epsilon \to 0} \epsilon^{(\sigma+\beta)(1-s)} \int_{A_1} \mu^1(B(x,\epsilon))^{s-1} {\rm d}\mu^1(x) = \infty,\]
and so, $D_{\mu^1}^-(s) \leq \beta$. Using the same reasoning as before, there is a closed interval $[a_2,b_2] =: I_{j_2} \subset A_1$ such that $|I_{j_2}| = \frac{1}{2} |A_1|$  and 
\[\limsup_{\epsilon \to 0} \epsilon^{(\sigma+\beta)(1-s)} \int_{I_{j_2}} \mu_\varphi^T(B(x,\epsilon))^{s-1} {\rm d}\mu_{\varphi}^T(x) = \infty\]
($I_{j_2}$ is ``one half'' of $A_1$). Then, define $A_2 := L_2 \cup I_{j_2} \cup  L'_2$, where $L_2 = [-|I_{j_2}|/4+a_2,a_2]$ and $L_2' = [b_2,b_2+|I_{j_2}|/4]$. Again, 
it follows that $D_{\mu^2}^-(s) \leq \beta$, where $\mu^2(\cdot):=\mu_\varphi^T(\cdot\cap A_2)$; note that
\[|A_2| = \frac{1}{2}|I_{j_2}| + |I_{j_2}| =  \frac{3}{2}|I_{j_2}| = \frac{3}{4} |A_1| = \frac{3}{4} \frac{3r}{2}.\]

Proceeding in this way, one builds a decreasing sequence of closed intervals $A_{k+1} \subset A_{k}$ such that $|A_k| \rightarrow 0$ as $k \rightarrow \infty$ (namely, $|A_k|=(3/4)^{k-1}(3r/2)$) and, for every $k \geq 1$, $D_{\mu^k}^-(s) \leq \beta$, with $\mu^k(\cdot) := \mu_{\varphi}^T(\cdot \cap A_k)$. Since each set $A_{k}$ is a compact interval, there exists $\Gamma \in [-r,r]$ such that $A_{k} \downarrow \{\Gamma\}$. 

Finally, for every $\psi \in \mathcal{H}$ and every $k \geq 1$, set $\psi_k := P^T({\mathbb{R}}\backslash A_k)\psi + \frac{1}{k} \varphi $. Since $T$ has purely continuous spectrum, $\displaystyle\lim_{k \to \infty}\Vert\psi_k-\psi\Vert=0$. Now, one has, for every $k \geq 1$, every $0<\epsilon<1 $ and every $x\in\mathbb{R}$,
\begin{eqnarray*}
\mu_{\psi_k}^T(B(x,\epsilon)) &\geq& \mu_{\psi_k}^T(B(x,\epsilon) \cap A_k)\\ &\geq& \frac{2}{k} {\rm Re} \langle P^T(B(x,\epsilon) \cap A_k)  P^T({\mathbb{R}}\backslash A_k)\psi , \varphi \rangle + \frac{1}{k^2} \mu_{\varphi}^T(B(x,\epsilon) \cap A_k)\\ &=& \frac{1}{k^2} \mu_{\varphi}^T(B(x,\epsilon) \cap A_k) = \frac{1}{k^2}\mu^k(B(x,\epsilon)),
\end{eqnarray*}
from which follows that 
\[\frac{\ln\left[ \frac{1}{\epsilon}\int \mu_{\psi_k}^T(B(x,\epsilon))^s \mathrm{d}x \right]}{(s-1)\ln \epsilon}  \leq  \frac{\ln \left[\frac{1}{k^{2s}} \frac{1}{\epsilon} \int \mu^k(B(x,\epsilon))^s \mathrm{d}x \right]}{(s-1)\ln \epsilon};\]
thus, for every $k \geq 1$, by Proposition~\ref{GFDproposition} {\em 1.}, $D_{\psi_k}^-(s) \leq  D_{\mu^k}^-(s) \leq  \beta$. Since $\psi$ is arbitrary, $\{\xi \in {\mathcal{H}} \mid D_{\mu_\xi^T}^-(s) \leq  \beta\}$ is dense in $\mathcal{H}$ and so, by Proposition~\ref{propgdelta}, it is a dense $G_\delta$ set in $\mathcal{H}$. 

\

\noindent {\em 2.} Since the case $\alpha=0$ is trivial we let $\alpha >0$. Let $0 \neq \varphi \in \mathcal{H}$ be such that $D_{\mu_{\varphi}^T}^+(q) \geq \alpha$. Again, let $r>0$ be such that $\supp(\mu_\varphi^T) \subset [-r,r]$;  set $I_1 := [-r,0]$ and $I_2 := [0,r]$. Since $D_{\mu_{\varphi}^T}^+(q) \geq \alpha$, it follows that, for every $0<\sigma<\alpha$,   
\begin{eqnarray*}
\infty &=& \limsup_{\epsilon \to 0}  \epsilon^{(\alpha - \sigma)(1-q)}   \int_{-r}^r \mu_\varphi^T(B(x,\epsilon))^{q-1} {\rm d}\mu_{\varphi}^T(x)\\ &\leq&  \sum_{j=1}^2\limsup_{\epsilon \to 0} \epsilon^{(\alpha - \sigma)(1-q)}\int_{I_j} \mu_\varphi^T(B(x,\epsilon))^{q-1} {\rm d}\mu_{\varphi}^T(x).
\end{eqnarray*}
Thus, for some $j_1 \in \{1,2\}$, 
\[\limsup_{\epsilon \to 0} \epsilon^{(\alpha - \sigma)(1-q)} \int_{I_{j_1}} \mu_\varphi^T(B(x,\epsilon))^{q-1} {\rm d}\mu_{\varphi}^T(x) = \infty.\]
Let $B_1 := I_{j_1}$ and set 
$\lambda^1(\cdot) := \mu_{\varphi}^T(\cdot\cap B_1)$. Then, for every $\epsilon >0$, 
\begin{eqnarray*}
\int_{B_1} \lambda^1(B(x,\epsilon))^{q-1} {\rm d}\lambda^1(x)  &=& \int_{B_1} \lambda^1(B(x,\epsilon))^{q-1} {\rm d}\mu_{\varphi}^T(x)\\ \\ &=& \int_{B_1} \mu_{\varphi}^T(B_1 \cap B(x,\epsilon))^{q-1}{\rm d}\mu_{\varphi}^T(x)\\ &\geq& \int_{B_1} \mu_{\varphi}^T(B(x,\epsilon))^{q-1}{\rm d}\mu_{\varphi}^T(x),
\end{eqnarray*}
where we have used in the last inequality the fact that $0<q<1$. Thus, for every $0<\sigma<\alpha$,
\[\limsup_{\epsilon \to 0} \epsilon^{(\alpha - \sigma)(1-q)} \int_{B_1} \lambda^1(B(x,\epsilon))^{q-1} {\rm d}\lambda^1(x) = \infty\]
and so, $D_{\lambda^1}^+(q) \geq \alpha$. Proceeding in this way, we build a decreasing sequence of closed intervals $B_{k+1} \subset B_{k}$ so that $|B_k| \rightarrow 0$ as $k \rightarrow \infty$ and, for each $k \geq 1$, $D_{\lambda^k}^+(q) \geq \alpha$, where
$\lambda^k(\cdot) := \mu_{\varphi}^T(\cdot \cap B_k)$. Since each set $B_{k}$ is a compact interval, there exists a $\Gamma \in [-r,r]$ such that $B_{k} \downarrow \{\Gamma\}$. 

Now set, for every $k \geq 1$, $A_k := B_k \setminus \{\Gamma\}$, and note that  $A_{k} \downarrow \emptyset$. Moreover, using the same reasoning as before, it follows that for every $k \geq 1$, $D_{\mu^k}^+(q) \geq \alpha$, where $\mu^k(\cdot) := \mu_{\varphi}^T(\cdot\cap A_k)$. Namely, given that $0<q<1$, one has for every $k \geq 1$ and every $\epsilon>0$, 
\begin{eqnarray*}
\epsilon^{(\alpha - \sigma)(1-q)}\int_{A_k} \mu^k(B(x,\epsilon))^{q-1} {\rm d}\mu^k(x) &=& \epsilon^{(\alpha - \sigma)(1-q)} \int_{A_k} \mu^k(B(x,\epsilon))^{q-1} {\rm d}\lambda^k(x)\\ &\geq&  \epsilon^{(\alpha - \sigma)(1-q)} \int_{A_k} \lambda^k(B(x,\epsilon))^{q-1} {\rm d}\lambda^k(x)\\ &=&  \epsilon^{(\alpha - \sigma)(1-q)} \int_{B_k} \lambda^k(B(x,\epsilon))^{q-1} {\rm d}\lambda^k(x)\\ &-& \epsilon^{(\alpha - \sigma)(1-q)} \lambda^k(B(\Gamma,\epsilon))^{q-1} \lambda^k(\{\Gamma\}).  
\end{eqnarray*}
We also have that 
\[\lim_{\epsilon \to 0} \epsilon^{(\alpha - \sigma)(1-q)} \lambda^k(B(\Gamma,\epsilon))^{q-1} \lambda^k(\{\Gamma\}) \leq \lim_{\epsilon \to 0}  \epsilon^{(\alpha - \sigma)(1-q)}  \lambda^k(\{\Gamma\})^q =0\]
if $\Gamma$ is an atom, and that \[\epsilon^{(\alpha - \sigma)(1-q)}\lambda^k(B(\Gamma,\epsilon))^{q-1} \lambda^k(\{\Gamma\}) = 0,\]
otherwise. Hence,
\begin{eqnarray*}
\limsup_{\epsilon \to 0} \epsilon^{(\alpha - \sigma)(1-q)} \int_{A_k} \mu^k(B(x,\epsilon))^{q-1} {\rm d}\mu^k(x) = \infty.
\end{eqnarray*}

Finally, for every $\psi \in \mathcal{H}$ and every $k \geq 1$, define $\psi_k:=P^T({\mathbb{R}} \backslash  A_k)\psi + \frac{1}{k} \varphi$, and note that $\displaystyle\lim_{k \to \infty}\Vert\psi_k-\psi\Vert=0$. Now, one has, for every $k \geq 1$, every $0<\epsilon<1 $ and every $x\in\mathbb{R}$,
\[\mu_{\psi_k}^T(B(x,\epsilon)) \geq \mu_{\psi_k}^T(B(x,\epsilon) \cap A_k) \geq \frac{1}{k^2} \mu_{\varphi}^T(B(x,\epsilon) \cap A_k) = \frac{1}{k^2}\mu^k(B(x,\epsilon)),\]
from which follows that 
\[\frac{\ln \left[\frac{1}{\epsilon}\int \mu_{\psi_k}^T(B(x,\epsilon))^q \mathrm{d}x \right]}{(q-1)\ln \epsilon}  \geq \frac{\ln\left[\frac{1}{k^{2q}} \frac{1}{\epsilon} \int \mu^k(B(x,\epsilon))^q \mathrm{d}x \right]}{(q-1)\ln \epsilon};\]
thus, for every $k \geq 1$,  by Proposition~\ref{GFDproposition} {\em 1.}, $D_{\psi_k}^+(q) \geq  D_{\mu^k}^+(q) \geq \alpha$. Since $\psi$ is arbitrary, $\{\xi \in {\mathcal{H}} \mid D_{\mu_\xi^T}^+(q) \geq  \alpha\}$ is dense in $\mathcal{H}$ and so, by Proposition~\ref{propgdelta}, it is a dense $G_\delta$ set in $\mathcal{H}$. 

%%%%%%%%%%%%%%%%%%%%%%%%%%%%%%%%%%%%%%%%%%%%%%%%%%%%%%%%%--Proof--of--Theorem--%%%%%%%%%%%%%%%%%%%%%%%%%%%%%%%%%%%%%%%%%%%%%%%%%%%%%%%%%%%%%%%%%%%%%%%%%%%%%%%%%%%%%%%%%%%%%%%%%%%%%%%%%%%%%%%%%%%%%%%%%%%%%%%%%%%%%%%%%%%%%%%%%%%%%%%%%%%%%%%%%%%%%%%%%%%%%%%%%

\subsection{Proof of Theorem~\ref{mainproposion}} \label{proofThmDims}

We need the following claims. 

\

\noindent {\bf Claim I:} Let $T$ be a bounded self-adjoint operator in $\mathcal{H}$. If $T$ has purely absolutely continuous spectrum, then 
\[\mathcal{N}^+(T) := \{\psi \mid D_{\mu_{\psi}^T}^+(q) = 1, { \rm \, \, for \, \, all} \, \, q >0, q \ne 1\}\]
is a dense $G_\delta$ set in $\mathcal{H}$.  

\

\noindent {\bf  Claim II:} Let $T$ be a bounded self-adjoint operator in $\mathcal{H}$. If $T$ has pure point spectrum, then
\[\mathcal{N}^-(T) := \{\psi \mid D_{\mu_{\psi}^T}^-(q) = 0, { \rm \, \, for \, \, all} \, \,q >0, q \ne 1 \}\]
is a dense $G_\delta$ set in $\mathcal{H}$.  

\

Since $(X,d)$ is a separable space, $C_{{\rm ac}} \subset X$ (with the induced topology) is also separable. Let $\{T_{j}\}$ be a dense sequence in  $C_{{\rm ac}}$ (being, therefore, dense in~$X$, since $C_{{\rm ac}}$ is dense in~$X$ by hypothesis). By {\bf  Claim I},   
\[\mathcal{N}^+ := \displaystyle\bigcap_{j} \mathcal{N}^+(T_j)\]
is a dense $G_\delta$ set in $\mathcal{H}$. Then, by Proposition~\ref{propgdeltabis}, for every $ \psi \in \mathcal{N}^+$,
\[X^+(\psi) := \{T \mid D_{\mu_{\psi}^T}^+(q) = 1, { \rm \; for \, all} \; q >0, q \ne 1 \} \supset \cup_{j}\{T_{j}\}\]
is a dense $G_\delta$ set in~$X$.

Similarly, since $C_{{\rm p}} \subset X$ is a separable space, let $\{S_{j}\}$ be a dense sequence in  $C_{{\rm p}}$. By {\bf  Claim II},   
\[\mathcal{N}^- := \displaystyle\bigcap_{j} \mathcal{N}^-(S_j)\]
is a dense $G_\delta$ set in $\mathcal{H}$. Then, it follows again by Proposition~\ref{propgdeltabis} that, for every $ \psi \in \mathcal{N}^-$,
\[
\{T \mid D_{\mu_{\psi}^T}^-(q) = 0, { \rm \;\; for \, all} \;\;  q >0, q \ne 1\}\supset \cup_{j}\{S_{j}\}\]
is a dense $G_\delta$ set in~$X$.

Combining the previous results, it follows that for every $\psi \in \mathcal{M} := \mathcal{N}^+ \cap \mathcal{N}^-$,
\[\{T \mid D_{\mu_{\psi}^T}^-(q) = 0 {\rm \;\; and \;\;} D_{\mu_{\psi}^T}^+(q) = 1 { \rm \; for \, all} \; q >0, q \ne 1   \}\]
is a dense $G_\delta$ set in~$X$. Hence, as the functions $q\mapsto D_\mu^\mp(q)$ are nonincreasing, follows that, for each  $\psi \in \mathcal{M}$, the set  $X_{{01}}(\psi) = \{T\in X \mid D_{\mu_{\psi}^T}^-(q) =0$ and $D_{\mu_{\psi}^T}^+(q) =1$, for all  $q >0 \}$ is residual in~$X$.

Now, it remains to prove the claims.

\

\noindent  {\bf Proof of Claim I:} Let $0 \neq \psi \in \mathcal{H}$. It follows from Proposition~\ref{scexpproposition0} that for~$\mu_{\psi}^T$-a.e.\ $x$,  $d_{\mu_{\psi}^T}^-(x) = 1$. So, there exists $\Omega \subset \mathbb{R}$ such that $\mu_{\psi}^T(\Omega) = \Vert \psi \Vert^2$ and such that, for every $x \in \Omega$, $d_{\mu_{\psi}^T}^-(x) =1$. Now let, for every $\epsilon > 0$, $f_\epsilon : \Omega  \longrightarrow \mathbb{R}$ be the measurable function given by the law
\[f_\epsilon(x) := \displaystyle\inf_{\epsilon> r} \frac{\ln \mu_{\psi}^T (B(x,r))}{\ln r}.\]
The sequence $(f_\epsilon(x))$ converges pointwise to $d_{\mu_{\psi}^T}^-(x)$; so,  by Egoroff's Theorem, there exist Borel sets $S_k \uparrow \Omega$ such that, for every $k \geq 1$, $\mu_{\psi}^T(S_k^c) < 1/k$ and such that $\displaystyle\lim_{\epsilon \downarrow 0}f_\epsilon(x)=d_{\mu_{\psi}^T}^-(x)$ uniformly on $S_k$. 
But then, given $0<\sigma < 1$, there exists $0 < \epsilon_{\sigma,k} < 1$ such that, for every $0 < \epsilon < \epsilon_{\sigma,k}$ and for every $x \in S_k$,
\begin{equation*}
\mu_{\psi}^T(B(x, \epsilon)) \leq \epsilon^{1 - \sigma}.    
\end{equation*}
Now, set $\psi_k:=P^T(S_k)\psi$ and note that: 
\[\displaystyle\lim_{k \to \infty}\Vert\psi_k-\psi\Vert^2=\lim_{k \to \infty}\Vert P^T(S_k^c)\psi\Vert^2\leq \lim_{k \to \infty} 1/k=0;\] 
for every $k \geq 1$, every $q>1$ and every $0<\epsilon<1$,

\[\frac{\ln \int \mu_{\psi_k}^T(B(x,\epsilon))^{q-1} {\rm d}\mu_{\psi_k}^T(x)  }{(q-1)\ln \epsilon} \geq \frac{2 \ln \Vert \psi_k \Vert}{(q-1)\ln \epsilon} + (1 -  \sigma),\]
from which follows that $D_{\mu_{\psi_k}^T}^\mp(q)=1$, since $0 < \sigma \leq 1$ is arbitrary. Hence, by Proposition~\ref{propgdelta},
\[\mathcal{N}^+(T) = \{\psi \mid D_{\mu_{\psi}^T}^+(q) = 1 { \rm \, \, for \, \, all}  \, \, q>1, \, q \in \mathbb{N}\},\]
 is a dense $G_\delta$ set in $\mathcal{H}$ (note that the above equality holds because $r \mapsto D_{\mu}^{+}(r)$ is a nonincreasing function). 

\

\noindent {\bf Proof of Claim II:} We note that this claim has been proven in Theorem 3.1. in \cite{ACdeOLMP}.  For the convenience of the reader, we present its proof in details. Let  $(\eta_j)$ be an orthonormal family of eigenvectors of~$T$, that is, $T\eta_j = \lambda_j \eta_j$ for every $j\ge 1$. Let, for every $0<q<1$, $(b_j) \subset \mathbb{C}$ be a sequence such that $|b_j| >0$, for all $j \geq 1$,  and $\sum_{j=1}^\infty \vert b_j \vert^{2q} < \infty$.  Given $\psi \in \mathcal{H}$, write $\psi = \sum_{j=1}^\infty a_j \eta_j$, and then consider, for each $k \geq 1$, 
\[\psi_k := \displaystyle\sum_{j=1}^k a_j \eta_j + \displaystyle\sum_{j=k+1}^\infty b_j \eta_j.\]
It is clear that $\psi_k \rightarrow \psi$. Moreover, for  $k \geq 1$ and each $\epsilon>0$,
\begin{eqnarray*}\label{theoGFD1}
\nonumber \int \mu_{\psi_k}^T(B(x,\epsilon))^{q-1} {\mathrm d}\mu_{\psi_k}^T(x) &=& \displaystyle\sum_{j=1}^\infty \mu_{\psi_k}^T(B(\lambda_j,\epsilon))^{q-1} \mu_{\psi_k}^T (\{\lambda_j\})\\ &\leq& \displaystyle\sum_{j=1}^\infty \mu_{\psi_k}^T (\{\lambda_j\})^q = \displaystyle\sum_{j=1}^k \vert a_j \vert^{2q} + \displaystyle\sum_{j=k+1}^\infty \vert b_j \vert^{2q},
\end{eqnarray*}
from which follows that $D_{\mu_{\psi_k}^T}^\mp(q) = 0$. Hence, for every $0<q<1$, $\{\psi \mid D_{\mu_{\psi}^T}^-(q) = 0\}$
is a dense set in $\mathcal{H}$; so, by Proposition~\ref{propgdelta},
\[\mathcal{N}^-(T) = \{\psi \mid D_{\mu_{\psi}^T}^-(q) = 0 { \rm \, \, for \, \, all} \, \, 0 < q < 1; \, q \in \mathbb{Q}\}\]
is a dense $G_\delta$ set in $\mathcal{H}$ (again, the above equality holds because $r \mapsto D_{\mu}^{-}(r)$ is a nonincreasing function).

%%%%%%%%%%%%%%%%%%%%%%%%%%%%%%%%%%%%%%%%%%%%%%%%%%%%%%%%%%%%%%%%%%%%%%%%%%%%%%%%%%%%%%%%%%%%%%%%%%%%%%%%%%%%%%%%%%%%%%%%%%%%%%%%%%%%%%%%%%
%%%%%%%%%%%%%%%%%%%%%%%%%%%%%%%%%%%%%%%%%%%%%%%%%%%%%%%%%%%--Acknowledgments--%%%%%%%%%%%%%%%%%%%%%%%%%%%%%%%%%%%%%%%%%%%%%%%%%%%%%%%%%%%%%%%%%%%%%%%%%%%%%%%%%%%%%%%%%%%%%%%%%%%%%%%%%%%%%%%%%%%%%%%%%%%%%%%%%%%%%%%%%%%%%%%%%%%%%%%%%%%%%%%%%%%%%%%%%%%%%%%%%%%%%%%%%%%%%%%%%%%%%%
 
\begin{center} \Large{Acknowledgments} 
\end{center}
\addcontentsline{toc}{section}{Acknowledgments}

MA thanks the partial support by CAPES (a Brazilian government agency). SLC thanks the partial support by FAPEMIG (Minas Gerais state agency; Universal Project under contract 001/17/CEX-APQ-00352-17) and CRdO thanks the partial support by CNPq (a Brazilian government agency, under contract 303503/2018-1).

\ 

\noindent {\bf Conflict of interest:} On behalf of all authors, the corresponding author states that there is no conflict of
interest.
%%%%%%%%%%%%%%%%%%%%%%%%%%%%%%%%%%%%%%%%%%%%%%%%%%%%%%%%%%%%%%%%%%%%%%%%%%%%%%%%%%%%%%%%%%%%%%%%%%%%%%%%%%%%%%%%%%%%%%%%%%%%%%%%%%%%%%%%%%%%%%%%%%%%%%%%%%%%%%%%%%%%%%%%%%%%%%%%%%%%%%%%%%%%%%%%%%%%%%--thebibliography--%%%%%%%%%%%%%%%%%%%%%%%%%%%%%%%%%%%%%%%%%%%%%%%%%%%%%%%%%%%
%%%%%%%%%%%%%%%%%%%%%%%%%%%%%%%%%%%%%%%%%%%%%%%%%%%%%%%%%%%%%%%%%%%%%%%%%%%%%%%%%%%%%%%%%%%%%%%%%%%%%%%%%%%%%%%%%%%%%%%%%%%%%%%%%%%%%%%%%%

\noindent  Email: moacir@ufam.edu.br, Departamento de Matem\'atica, UFAM, Manaus, AM, 69067-005 Brazil

\noindent  Email: silas@mat.ufmg.br, Departamento de Matem\'atica, UFMG, Belo Horizonte, MG, 30161-970 Brazil

\noindent  Email: oliveira@dm.ufscar.br,  Departamento  de  Matem\'atica,   UFSCar, S\~ao Carlos, SP, 13560-970 Brazil

\end{document}